\documentclass[letterpaper,10pt,reqno]{amsart}[1996/10/24]
\usepackage{lmodern}
\usepackage{tikz}
\usepackage{graphicx,color}
\usepackage{amsthm}
\usepackage{amsfonts,amssymb,mathtools}
\usepackage{subfig}
\usepackage{bbm}
 \usepackage[foot]{amsaddr}

\newtheorem{assumptionA}{Assumption}
\theoremstyle{plain}

\newcounter{count}[section]\numberwithin{count}{section}
\newtheorem{theorem}[count]{Theorem}

\newtheorem{proposition}[count]{Proposition}

\newtheorem{lemma}[count]{Lemma}
\newtheorem{remark}[count]{Remark}

\numberwithin{equation}{section}

\newcommand{\E}{\mathbb{E}}

\global\long\def\Esp#1{\mathbb{E}\left[#1\right]}

\global\long\def\V#1{\mathbb{V}\mathrm{ar}\left[#1\right]}

\global\long\def\Cov#1{\mathbb{C}\mathrm{ov}\left[#1\right]}

\global\long\def\ind{\mathbbm{1}}

\begin{document}

\title{Probabilistic and Piecewise Deterministic models in Biology}
\author{Bertrand Cloez$^*$}\address{$^*$MISTEA, INRA, Montpellier SupAgro, Univ. Montpellier, France.}
\author{Renaud Dessalles$^\dag$}\address{$^\dag$INRA, UR MaIAGE, domaine de Vilvert, Jouy en Josas, France.}
\author{Alexandre Genadot$^\ddag$}\address{$^\ddag$Institut de Math\'ematiques de Bordeaux, Inria Bordeaux-Sud-Ouest, France.}
\author{Florent Malrieu$^\S$}\address{$^\S$Laboratoire de Math\'ematiques et Physique Th\'eorique
	(UMR CNRS 7350), F\'ed\'eration Denis Poisson (FR CNRS 2964), Universit\'e
	Fran\c cois-Rabelais, Parc de Grandmont, 37200 Tours, France.}
\author{\\ Aline Marguet$^\bot$}\address{$^\bot$CMAP, \'Ecole Polytechnique, Universit\'e Paris-Saclay, Palaiseau, France.}
\author{Romain Yvinec$^\sharp$}\address{$^\sharp$Physiologie de la Reproduction et des Comportements, Institut 
	National de la Recherche Agronomique (INRA) UMR85, CNRS-Universit\'e Fran\c{c}ois-Rabelais UMR7247, IFCE, Nouzilly, 
	37380 France}
%
%
\begin{abstract} We present recent results on Piecewise Deterministic Markov Processes (PDMPs), involved in biological modeling. PDMPs, first introduced in the probabilistic literature by \cite{Davis1984}, are a very general class of Markov processes and are being increasingly popular in biological applications. They also give new interesting challenges from the theoretical point of view. We give here different examples on the long time behavior of switching Markov models applied to population dynamics, on uniform sampling in general branching models applied to structured population dynamic, on time scale separation in integrate-and-fire models used in neuroscience, and, finally, on moment calculus in stochastic models of gene expression.
	
\end{abstract}
%

%
\maketitle
\section*{Introduction}

The piecewise deterministic Markov processes (denoted PDMPs) were first introduced in
the literature by Davis \cite{Davis1984,Davis1993}. The key point of his work was to endow 
the PDMP with rather general tools, similar as such that already exist for diffusion processes. 
Indeed, PDMPs form a family of non-diffusive c\`adl\`ag Markov processes, involving a 
deterministic motion punctuated by random jumps. The motion of the PDMP 
$\{X(t)\}_{t\geq 0}$ depends on three local characteristics, namely the jump rate $\lambda$, 
the deterministic flow $\varphi$ and the transition measure $Q$ according to which the 
location of the process at the jump time is chosen. The process starts from $x$
and follows the flow $\varphi(x, t)$ until the first jump time $T_1$ which occurs either 
spontaneously in a Poisson-like fashion with rate $\lambda(\varphi(x, t))$ or when the 
flow $\varphi(x, t)$ hits the boundary of the state-space. In both cases, the location of the 
process at the jump time $T_1$, denoted by $Z_1 = X(T_1)$, is selected by the transition 
measure $Q(\varphi(x, T_1), \cdot)$ and the motion restarts from this new point as before. 
This fully describes a piecewise continuous trajectory for $\{X(t)\}_{t\geq 0}$ with jump 
times $\{T_k\}$ and post jump locations $\{Z_k \}$, and which evolves according to the 
flow $\varphi$ between two jumps.

Since the seminal work of Davis, PDMPs have been heavily studied from the theoretical 
perspective, we may refer for instance the readers to 
\cite{Costa2008,Bakhtin2012,Azais2014,M15,BBMZ12,BBMZ15} among many others. 

From the applied point of view, and more precisely in biological applications, these 
processes are sometimes referred as hybrid, and are especially appealing for their 
ability to capture both continuous (deterministic) dynamics and discrete (probabilistic 
or deterministic) events.  First applications date back at least to the
studies of the cell cycle model
\cite{Lasota1992,Lasota1999}, and more recent works, to name just a few, in neurobiology 
\cite{PTW,DL15,DL15},  cell population and branching models  
\cite{BDMV,cloez,marguet2016,DHKR,L10,CH16}, gene expression models 
\cite{YZLM,Mackey2013,DFR}, food contaminants \cite{B15,BCT} or multiscale chemical reaction 
network models \cite{Crudu2012,Ball2006,Kurtz2013,Hepp2015}. See also 
\cite{Kouretas2006,Bressloff2014,Rudnicki2015} for selected reviews on hybrid models in biology.

The paper is organized as follows. In Section 1, we present long time behavior results 
for switched flows in population dynamics. In Section 2,  we present many-to-one formulas 
and description of the trait of an individual uniformly sampled in a general branching 
model applied to structured population dynamic. In Section 3, we present limit theorems 
and time scale separation in integrate-and-fire models used in neuroscience.  Finally, in 
Section 4, we derive asymptotic moments formulas in a stochastic model of gene expression.


\section{Long time behavior of some PDMP for population dynamics}
 
 In this section, we present recent results on long time behavior and exponential convergence 
 of some PDMP used in population dynamics, in particular, for modeling population growth in 
 a varying environment.

We consider processes $(X_t,I_t)_{t\geq 0}$, where the first component $X_t$ has continuous 
paths in a continuous space $E$, namely a subset of $\mathbb{R}^d$, for some integer $d$, 
and  the second component $I_t$ is a pure jump process on a finite state space $F$ or in 
$\mathbb{N}$.  
The continuous variable represents the number of individuals (or more precisely the 
population density or the relative abundance) of a population and the discrete variable 
models the population environment. In a fixed environment $I_t=i\in F$, we assume 
that $(X_t)_{t\geq 0}$ evolves as the solution of an ordinary differential equation, that is
\begin{equation}
\label{eq:ODE}
\forall t\geq 0\,,  \quad \partial_t X_t = F^{(i)} (X_t),
\end{equation}
where $F^{(i)}$ is some smooth function. For instance, the oldest and maybe the most 
famous model to represent the evolution of a population is the choice $E=\mathbb{R}_+$ and
\begin{equation}
\label{eq:lin-malthus}
F^{(i)} : x\mapsto a_i x, \quad a_i \in \mathbb{R}.
\end{equation}
When there is no variability in the environment, solutions are given by exponential functions and there is either explosion or extinction according to the sign of $a_i$. We develop in Subsection \ref{sect:malthus} this toy model when the environment varies. Almost all properties (almost-sure convergence, moments...) can be derived and it gives a first hint to understand more complicated models such as multi-dimensional linear models or non-linear models. These generalizations are described in Subsection \ref{secr:linear}, and \ref{sect:general}, respectively. For non-linear population models, maybe one of the most famous models is given by the Lotka-Volterra equations (also known as the predator-prey equations), with state-space $E= \mathbb{R}_+^2$ and
$$
F^{(i)}(x,y) = \begin{pmatrix}
\alpha_i x (1- a_i x- b_i y)\\
\beta_i x (1- c_i x- d_i y)\\
\end{pmatrix}, \quad a_i,b_i,c_i,d_i,\alpha_i,\beta_i \in \mathbb{R}_+.
$$

To end the introduction of this section, let us mention that all of these models are particular instance of a class of switching Markov model, see for instance \cite{BBMZ15,YZ10} for general references. Some other examples of application are developed in \cite{M15,FGM, Mon16}.

\subsection{Malthus model}
\label{sect:malthus}

Let us consider, in this subsection, that $(I_t)_{t\geq 0}$ is an irreducible continuous time Markov chain on a finite state space $F$. We denote by $\nu$ its unique invariant probability measure.  The process $(X_t)_{t\geq0}$ will be the solution of
$$
\forall t\geq 0, \ \partial_t X_t = a_{I_t} X_t.
$$
This means that $F^{(i)}$ is given by  \eqref{eq:lin-malthus}. Thus, we have
$$
\forall t\geq 0, \ X_t = X_0 e^{\int_0^t a_{I_s} ds}.
$$
The almost-sure behavior of this process is easy to understand. Indeed, by the ergodic theorem, we have that
$$
\lim_{t\to + \infty } \frac{1}{t} \int_0^t a_{I_s} ds = \sum_{i\in F} a_i\nu(\{i\}). = \nu(a).
$$
Then, we have the following dichotomy: either $\nu(a)>0$ and $(X_t)_{t\geq 0}$ tends almost-surely to infinity at an exponential rate, or $\nu(a)<0$ and $(X_t)_{t\geq 0}$ tends to $0$. The second statement can be understood as the extinction of the population, even if, in contrast with classical birth and death processes, $(X_t)_{t\geq 0}$ never hits $0$; that is the extinction time (or the hitting time of $0$) is almost-surely infinite. 
\begin{figure}
	\begin{center}
	\includegraphics[scale=0.5]{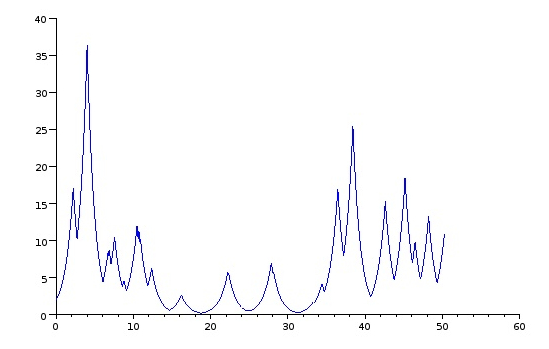}
	\end{center}
	\caption{Trajectory of $(X_t)_{t\geq 0}$ in the Malthus model.}
	\label{fig:malthus}
\end{figure}

The convergence of moments of $(X_t)_{t\geq 0}$ is more tricky than the almost-sure convergence. Indeed, one cannot use a dominated convergence theorem or a Jensen inequality (which is in the wrong way). Nevertheless, we have

\begin{lemma}[Convergence of moments]
	If $\nu(a)>0$ then 
	$$
	\forall q>0 , \quad \lim_{t\to +\infty} \mathbb{E}[X_t^q] =+\infty.
	$$
	If $\nu(a)<0$ then there exists $p>0$ such that 
	$$
	\forall q\leq p, \quad \lim_{t\to +\infty} \mathbb{E}[X_t^q] =0.
	$$
\end{lemma}
The proof which follows comes from \cite{BGM}; see also \cite{SY05}.
\begin{proof}
	If we denote by $A$ the generator of the process $(I_t)_{t\geq 0}$ and $\mu_0$ the vector corresponding to its initial law, then the Feynman-Kac formula states that, for every $p>0$, 
	$$
	\mathbb{E}[X_t^p]= \mathbb{E}[X_0^p]\mathbb{E}\left[e^{\int_0^t p a_{I_s} ds} \right]= \mathbb{E}[X_0^p] \mu_0 e^{t(A+ pD)}\mathbf{1},
	$$
	where it is the classical exponential of matrices, $D$ is the diagonal matrix such that $D_{i,i}=a_i$ and $\mathbf{1}$ is the vector made of ones. Now, classical results on matrices (including Perron-Frobenius Theorem) give the existence of two constants $c_p, C_p>0$ such that 
	$$
	c_p  e^{\lambda_p t} \leq \mu_0 e^{t(A+pD)} \mathbf{1}  \leq C_p e^{\lambda_p t},
	$$
	where $\lambda_p$ is the (unique) eigenvalue of largest real part of the matrix $(A+pD)$. It then remains to understand the sign of $\lambda_p$ to understand the moment behaviour. By definition of $\lambda_p$,  there exists $v_p$ such that
	$$
	v_p (A+pD)  = \lambda_p v_p.
	$$  
	Moreover, for $p=0$, we have $\lambda_0=0$ and $v_0 =\nu$. Again classical results on matrices ensure that $p \mapsto \lambda_p$ and $p\mapsto v_p$ are smooth and then differentiating in $p$, we obtain
	$$
	v_p D  +  \partial_p v_p (A+pD) = v_p \partial_p  \lambda_p + \lambda_p \partial_p v_p. 
	$$
	In particular, for $p=0$ and multiplying by $\mathbf{1}$ by the right, we have $ \nu(a) = (\partial_p \lambda_p)_{\scriptscriptstyle{\vert p=0}}$. As a consequence, if $\nu(a)>0$, then $p\mapsto \lambda_p$ is increasing in a neighborhood of  $0$ and for all $p>0$ small enough we have $\lambda_p>0$. Thus, $\mathbb{E}[X_t^p]$ tends to infinity for $p$ small enough, and then for all $p>0$ by the Jensen inequality. In contrary, if $\nu(a)<0$ then there exists $p>0$ such that $\lim_{t \to +\infty} \mathbb{E}[X_t^p]=0$ and also for all $q<p$, again by Jensen inequality.
\end{proof}

Understanding the long time behavior of moments is primordial when we are interested in the convergence to some invariant law more general than $\delta_0$. Indeed, moments provide Lyapunov-type functionals ensuring that the process returns frequently in a compact-set.

\subsection{Others linear models}
\label{secr:linear}

From a modeling point of view, the results of the previous section are easily understandable. Indeed, it is enough that the environment has in mean a positive effect to ensure an exponential growth, while if the environment is, in mean, not favorable, then the population goes to extinct.

However, note that this interpretation only holds because the growth parameter is well-chosen. Indeed, let us first consider a discrete time analogous of the Malthus model: the sequence $(Y_n)_{n\geq 0}$ satisfying
$$
\forall n\geq 0, \quad  Y_{n+1} = \Theta_n Y_n,
$$
where $(\Theta_n)_n$ is a sequence of i.i.d. non-negative random variables.
Here $\Theta_n$ represents the mean number of  children per individuals at generation $n$ (all individuals have the same number of children which depends on the environment). If $(\Theta_n)_{n\geq 0}$ is a constant sequence then $Y_n = \Theta_1^n Y_0$ and a dichotomy occurs with respect to a critical value $\Theta_1=1$. When $(\Theta_n)_{n\geq 0}$ is a sequence of i.i.d random variable, then there is also a dichotomy but the critical value (to be smaller or larger than $1$) is no longer $\mathbb{E}[\Theta_1]$ but  $\exp\left(\mathbb{E}[\ln(\Theta_1)]\right)$. Indeed, we have
$$
\forall n\geq 0, \quad  Y_{n} = Y_0 \prod_{k=0}^{n-1} \Theta_k = X_0 e^{\sum_{k=0}^{n-1} \ln(\Theta_k)}.
$$
This is exactly the same phenomenon for Galton-Watson chain in random environment, see for instance \cite{BV17} and \cite[Section 2.9.2]{HPV}.

There is a similar (but more complex, maybe) problem in upper dimension. Let us consider that $E=\mathbb{R}^2$, $F=\{0,1\}$, and $(I_t)_{t\geq 0}$ jumps from $0$ to $1$ and from $1$ to $0$ with the same rate $\lambda>0$, and
\begin{equation}
\label{eq:plan}
F^{(0)} : v \mapsto \begin{pmatrix}
-1 & 4 \\
-1/4  & -1
\end{pmatrix} \cdot v, \quad F^{(1)}:v\mapsto \begin{pmatrix}
-1 & -1/4 \\
4  & -1
\end{pmatrix} \cdot v.
\end{equation}
Then the solutions of the equation $\partial_t \begin{pmatrix}
x \\
y
\end{pmatrix} = F^{(1)} \begin{pmatrix}
x \\
y
\end{pmatrix}$ are given by
\begin{equation}
\label{eq:solution}
\forall t\geq 0, \quad 
\left \{
\begin{array}{c @{=} c}
x(t) & e^{-t} (\cos(t)x(0) - \sin(t)y(0)/4)  \\
y(t) & e^{-t} (4\sin(t) x(0) + \cos(t) y(0)). \\
\end{array}
\right.
\end{equation}
It is easy to see that there exists some constants $C>0$ such that for all $t\geq 0$, 
$$
\Vert (x(t),y(t))\Vert \leq Ce^{-t} \Vert (x(0),y(0))\Vert.
$$
This property is also satisfied by the solutions of  $\partial_t \begin{pmatrix}
x \\
y
\end{pmatrix} = F^{(0)} \begin{pmatrix}
x \\
y
\end{pmatrix}$. However, we have the following theorem.

\begin{theorem}[\cite{BBMZ14}]
	\label{th:plan}
	Let $(X_t)_{t\geq 0}$ be the solution of Equation \eqref{eq:ODE} with flows defined in \eqref{eq:plan}.
	There exists $\beta_c>0$ such that if $\lambda<\beta_c$ then 
	$$
	\lim_{t \to + \infty} \Vert X(t) \Vert =0,
	$$
	and if $\lambda>\beta_c$ then $\lim_{t \to + \infty} \Vert X(t) \Vert = + \infty$. Moreover, the speed of convergence is exponential.
\end{theorem}

For this type of results for more general vector fields, see for instance \cite{BBMZ14, M15, LMR}. 
Let us briefly explain this phenomenon. One can see that, whatever the initial conditions of 
the solutions  $(x,y)$ of \eqref{eq:solution}, the map $t\mapsto \| (x(t), y(t)) \|$ is not decreasing; 
see for instance Figure~\eqref{fig:distance}.

\begin{figure}
	\begin{center}
		\includegraphics[scale=0.5]{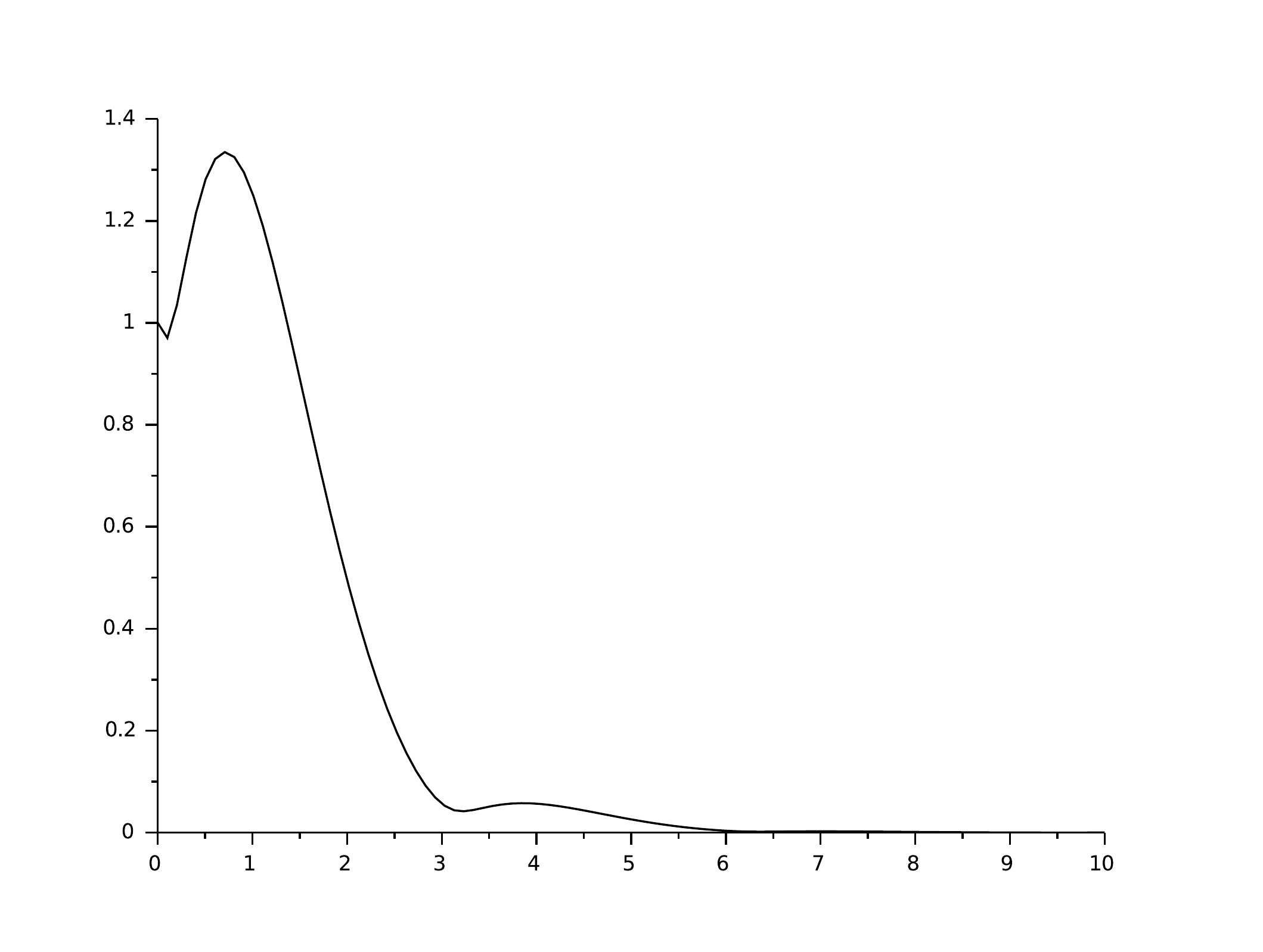}
	\end{center}
	\caption{Graph of $t\mapsto \| (x(t) , y(t)) \| = e^{-t} \sqrt{15 \sin^2(t) +1}$ for \eqref{eq:solution} when $(x_0,y_0)=(1,0)$.}
	\label{fig:distance}
\end{figure}

Then, heuristically, we see that if $I$ jumps sufficiently fast (before the decay of the distance) then the distance may grow. However, if $I$ is too slow, then $(X_t)_{t\geq 0}$ essentially follows one of the two flows and goes rapidly to $0$.

Switching linear models can have even more unpredictable behavior than the previous one. Indeed, in \cite{LMR}, the author exhibits two different linear functions $F^{(0)}, F^{(1)}$ on $\mathbb{R}^2$ and two critical values $\beta_2>\beta_1$ such that the switched process $(X_t)_{t\geq 0}$ tends to infinity for at least one $\lambda \in (\beta_1,\beta_2)$ and to $0$ if $\lambda \in (0, \beta_1) \cup (\beta_2 , + \infty)$.

\subsection{A general theorem for convergence}
\label{sect:general}

In this subsection, we expose a condition similar to Subsection \ref{sect:malthus} for ensuring convergence to an equilibrium (not necessary $\delta_0$). As pointed out in the previous section, we have to measure precisely the growth of each underlying flow. So we assume that for all $i\in F$, there exists $\rho(i) \in \mathbb{R}$ such that
\begin{equation}
\label{eq:coercive}
\langle x-y, F^{(i)}(x) - F^{(i)} (y) \rangle \leq - \rho(i) \Vert x- y \Vert^2.
\end{equation}
The parameter $\rho(i)$ measures the contraction of a solution $\dot{u} = F^{(i)}(u)$ with respect to the norm $\Vert\cdot \Vert$. Indeed, for two solutions, $u^1$ and $u^2$, we have 
$$
\forall t\geq 0 , \quad \Vert u^1(t) - u^2(t) \Vert^2 \leq e^{- \rho(i) t } \Vert u^1(0) - u^2(0) \Vert^2.
$$
Note that the next result does not rely on a specific norm, but it has to be the same for all flows. A sufficient condition to prove \eqref{eq:coercive} is given by the next lemma.


\begin{lemma}
	Let $G: \mathbb{R}^d \to \mathbb{R}^d$ be a $C^1$ map and let $\mathbf{J}_x G$ be its 
	Jacobian matrix at the point $x\in \mathbb{R}^d$. If for all $u,x\in \mathbb{R}^d$,
	$$
	\langle u, \mathbf{J}_x G \cdot u \rangle \leq - \rho \Vert u \Vert^2,
	$$
	then, for all $x,y \in \mathbb{R}^d$,
	\begin{equation}
	\label{eq:coercivebis}
	\langle x-y, G(x) - G (y) \rangle \leq - \rho \Vert x- y \Vert^2.
	\end{equation}
	In particular, if $G=\nabla V$, for some $C^2$ function $V$, then \eqref{eq:coercivebis} holds under the classical condition that $V$ is strongly convex with constant $\rho$; namely for all $x\in \mathbb{R}^2$,
	$$
	\textrm{Hess}_x V \geq \rho I,
	$$
	in the sense of quadratic form.
\end{lemma}
\begin{proof}
	It is classic to derive the previous bounds using the mean value theorem for the following map
	$$
	\varphi: t\mapsto \frac{\langle x-y, G(x) - G (ty +(1-t)x) \rangle}{\Vert x- y \Vert^2}.
	$$
\end{proof}

We can now state the main result of this subsection.
\begin{theorem}[\cite{CH15}]
	\label{th:general}
	Assume that $(I_t)_{t\geq 0}$ is an irreducible Markov process on a finite state space $F$ with invariant measure $\nu$, $(X_t)_{t\geq 0}$ satisfies \eqref{eq:ODE}, and that Eq.~\eqref{eq:coercive} holds. If
	$$
	\sum_{i\in F} \rho(i) \nu(i)> 0,
	$$
	then $(X_t,I_t)$ converges in law to a unique invariant law.
\end{theorem}

The proof of this theorem is given in \cite{CH15,BBMZ12}. The proof of \cite{CH15} is based on a generalization (developed in \cite{HMS}) of the classical Harris Theorem \cite{HM11}, which does not hold in general for PDMP (see for instance the introduction of \cite{BBMZ12}). One of the main point of the proof is the construction of a Lyapunov function $V$, in the sense that there exist, $C,\gamma,K>0$ such that
\begin{equation}
\label{eq:lyapunov}
\forall t\geq0, \ \mathbb{E}[V(X_t,I_t)] \leq Ce^{-\gamma t} V(X_0,I_0) +K.
\end{equation}
The construction of $V$ and the control of its moments are based on Lemma 1.1.

Note that another proof of the previous theorem is given in \cite{BBMZ12}. Nevertheless, the proof of \cite{CH15} takes the advantage to be generalizable in many contexts (dependent environment, non-deterministic underlying dynamics, asymptotic assumptions...). See \cite{CH15} for details.

Closely related articles \cite{BL14,MH16,MZ16} deal with a fluctuating Lotka-Volterra model. They show that random switching between two environments both favorable to the same specie can lead to the extinction of this specie or coexistence of both species.

As Theorem \ref{th:general}, their proofs are based on the construction of a Lyapunov function. Nevertheless instead of controlling the exit from a compact set of $\mathbb{R}$ to infinity, this Lyapunov function is used to control the exit from a compact set included in $(0, \infty)^2$ to the axes $\{(0,y) | y\geq0\}$ or $\{(x,0,) | x\geq 0 \}$.  A complete description of this model is given in \cite{MH16}. Note however that their counter-intuitive result is similar to the one of Theorem \ref{th:plan}. It is based on the fact that, when $I$ jumps sufficiently fast (namely $\lambda$ being large in Theorem \ref{th:plan}), then $X$ is close to follow the following deterministic dynamics:
$$
\forall t\geq 0, \quad \partial_t X_t= \sum_{i\in F} \nu(i) F^{(i)} (X).
$$


Finally, we end this section mentioning that such hybrid framework can also be applied to models where the population is represented by a discrete variable, and/or the environment fluctuates continuously. One can cite for instance \cite{C15}, where the author studies a model of interaction between a population of insects and a population of trees. Another example is the chemostat model (a type of bioreactor), introduced in \cite{CY} and studied in \cite{CMMM,CJL,CF14,CF16}.
Lyapunov functions are again a key theoretical tool to study long time behavior for such models. However, we point out that extinction events may occur in discrete population models, and that the interplay between the discrete and continuous components makes the problem more difficult than standard absorbing time studies in purely discrete population models. Nevertheless, it was proved in \cite[Theorem 3.4]{CF16} that extinction event is almost-surely finite and has furthermore exponential tails. This generalizes part of a result of \cite{CMMM}.

\section{Uniform sampling in a branching structured population}
In this section, which is a short version of \cite{marguet2016}, we give a characterization of the trait of a uniformly sampled individual among the population at time $t$ in a structured branching population. 

\subsection{Introduction}
We consider a branching Markov process where each individual $u$ is characterized by a trait $\left(X_t^u,\ t\geq 0\right)$ which dynamic follows an $\mathcal{X}$-valued Markov process, where $\mathcal{X}:=\widetilde{\mathcal{X}}\times \mathbb{R}_+$ and $\widetilde{\mathcal{X}}\subset\left(\mathbb{R}_+\right)^{d}$ for some $d\geq 1$. For any $x\in\mathcal{X}$, the last coordinate corresponds to a time coordinate and we denote by $\widetilde{x}\in\widetilde{\mathcal{X}}$ the vector of the first $d$ coordinates. We assume that this trait influences the lifecycle of each individual (in terms of lifetime, number of descendants and inheritance of the trait). Then, an interesting problem consists in the characterization of the trait of a "typical" individual in the population. This is the question we address here.

Let us describe the process in more details. We consider a strongly continuous contraction semi-group with associated infinitesimal generator $\mathcal{G}:\mathcal{D}(\mathcal{G})\subset\mathcal{C}_b(\mathcal{X})\rightarrow\mathcal{C}_b(\mathcal{X})$, where $\mathcal{C}_b(\mathcal{X})$ denotes the space of continuous bounded functions from $\mathcal{X}$ to $\mathbb{R}$. Then, each individual $u$ has a trait $(X_t^u,\ t\geq 0)$ which evolves as a Markov process defined as the unique $\mathcal{X}$-valued c\`adl\`ag solution of the martingale problem associated with $(\mathcal{G},\mathcal{D}(\mathcal{G}))$. An individual with trait $x$ dies at an instantaneous rate $B(x)$, where $B$ is a continuous function from $\mathcal{X}$ to $\mathbb{R}_+$. It is replaced by $A_{u}(x)$ children, where $A_{u}(x)$ is a $\mathbb{N}$-valued random variable  with distribution $\left(p_{k}\left(x\right),k\geq 0\right)$. For convenience, we assume that $p_1(x)\equiv 0$ for all $x\in\mathcal{X}$. The trait of the descendants at birth depends on the trait of the mother at death. For all $k\in\mathbb{N}$, let $P^{(k)}(x,\cdot)$ be the probability measure on $\mathcal{X}^k$ corresponding to the trait distribution at birth of the $k$ descendants of an individual with trait $x$. We denote by $P_{j}^{(k)}\left(x,\cdot\right)$ the $j$th marginal distribution of $P^{(k)}$ for all $k\in\mathbb{N}$ and $j\leq k$.

Finally, we denote by $\mathcal{M}_{P}(\mathcal{X})$ the set of point measures on $\mathcal{X}$. Following Fournier and M\'el\'eard \cite{fournier2004microscopic}, we work in $\mathbb{D}\left(\mathbb{R}_{+},\mathcal{M}_{P}\left(\mathcal{X}\right)\right)$, the state of c\`adl\`ag measure-valued Markov processes. For any $Z\in \mathbb{D}\left(\mathbb{R}_{+},\mathcal{M}_{P}\left(\mathcal{X}\right)\right)$, we write:
\[
Z_t=\sum_{u\in V_t} \delta_{X_t^u},\ t\geq 0,
\]
the measure-valued process describing the dynamic of the population where $V_t$ denotes the set of all individuals in the population at time $t$. We will denote by $N_t$ the cardinality of $V_t$ and by:
\begin{align*}
	m(x,s,t)=\mathbb{E}\left(N_t\big|Z_s=\delta_x\right),
\end{align*}
its expected value, for $0\leq s \leq t$.

An example of such a process is the size-structured population where each individual grows exponentially fast. For more details we refer the reader to \cite{DHKR} or \cite{bertoin2016}. Then, at rate $B(X_t^u)$, the individual $u$ splits at time $t$ in two daughter cells of size $X_t^u/2$.
\begin{figure}
	\includegraphics[scale=0.5]{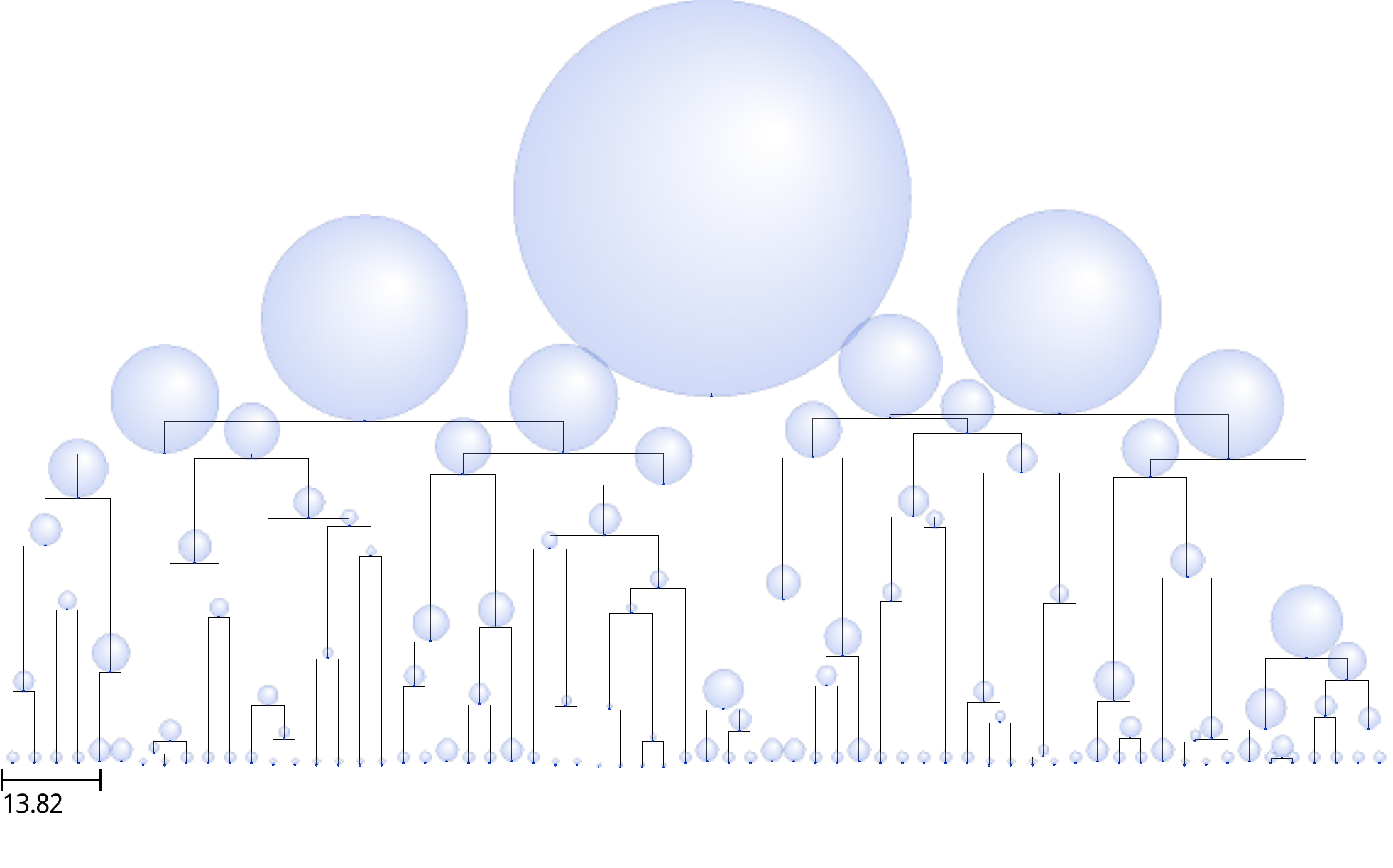}
	\caption{Descending genealogy from an individual with size $1$ until time $T=50$ of a size-structured population. Each cell grows exponentially at rate $0.01$ and divide at rate $B(x)=x$.}
	\label{fig:arbre}
\end{figure}
Figure \ref{fig:arbre} is a realization of such a process. The diameter of each circle corresponds to the size of the individual at division and the length of the branch represents the lifetime of each individual. In particular, we notice the link between the lifetime and the size of each individual: the bigger the cell is, the shorter its lifetime is.

In Section \ref{sec:exis}, we give the chosen assumptions on the model to ensure that our process is well-defined. Then, in Section \ref{sampling}, we describe the process corresponding to the trait of a "typical" individual in the population. Finally, in Section \ref{sec:large}, we explain why this process corresponds  to the trait of a uniformly sampled individual in a large population approximation. 

\subsection{Definition and existence of the structured branching process\label{existence}}\label{sec:exis}

To define rigorously the structured branching process on $\mathbb{R}_+$, we need to ensure that almost surely the number of individuals in the population does not blow up in finite time. 

As explained in the previous description of the process, the lifetime of each individual depends on the dynamic of the trait through the division rate $B$. One way of dealing with this dependency is to assume that the division rate is upper bounded by some constant $\overline{B}$. Then, in the case of a binary division, the expected number of individuals in the population at time $t$ is upper bounded by the expected value of a Yule process with birth rate $\overline{B}$ at time $t$ which is equal to $\exp(\overline{B}t)$ and which is bounded on compact sets. In the case of an unbounded division rate, the same reasoning applies if we can control the excursions of the dynamic of the trait. 
Thus, we consider two sets of hypotheses: the first one controls what happens regarding divisions (in term of rate of division and of mass creation) and the second one controls the dynamic of the trait between divisions.
\begin{assumptionA}\label{assu:debut1}
	We consider the following assumptions:
	\begin{enumerate}
		\item There exist $b_1,b_2\geq 0$ and $\gamma\geq 0$ such that for all $x\in\mathcal{X}$,
		\[B(x)\leq b_1\left| x\right|^{\gamma}+b_2.\]
		\item There exists $\underbar{x}\in\widetilde{\mathcal{X}}$ such that	for all $x\in\mathcal{X}$, $k\in\mathbb{N}$:
		\[
		\int_{\mathcal{X}^k}\left(\sum_{i=1}^{k}\widetilde{y}_i\right)P^{(k)}(x,dy_1\ldots dy_k)\leq \widetilde{x}\vee \underline{x},\text{ componentwise}.
		\]		
		\item
		There exists $\overline{m}\geq 0$ such that for all $x\in \mathcal{X}$, 
		\[ m(x):=\sum_k kp_k(x)\leq \overline{m}.\]
		
	\end{enumerate}
\end{assumptionA}
The first point controls the division rate. The second point means that we consider a fragmentation process with a possibility of mass creation at division when the mass is small enough. In particular, clones are allowed in the case of bounded traits and bounded number of descendants and any finite type branching structured process can be considered.

As explained before, we make a second assumption to control the behavior of traits between divisions.
\begin{assumptionA}\label{assu:debut2}
	There exist $c_1,c_2\geq 0$ such that for all $x\in\mathcal{X}$:
	\begin{align*}
		\mathcal{G}h_{\gamma}(x)\leq c_1 h_{\gamma}(x)+c_2,
	\end{align*}
	where $\gamma$ is defined in Assumption $\ref{assu:debut1}$ and for $x\in\left(\mathbb{R}_+\right)^d$, $h_{\gamma}(x)=|x|^{\gamma}=\left(\sum_{i=1}^d x_i\right)^{\gamma}$.
\end{assumptionA}

This assumption ensures in particular that the trait does not blow up in finite time in the case of a deterministic dynamic for the trait. Assumptions \ref{assu:debut1}(1) and \ref{assu:debut2} are linked via the parameter $\gamma$ which controls the balance between the growth of the population and the dynamic of the trait. In particular, if Assumption \ref{assu:debut1} is satisfied for $\gamma=0$, the division rate is bounded and Assumption \ref{assu:debut2} is always satisfied. 

Under Assumptions \ref{assu:debut1} and \ref{assu:debut2}, the previously described measure-valued process is well-defined as the strongly unique solution of a stochastic differential equation. 

The proof of existence relies on a recursive construction of a solution via the sequence of jumps time of the population. Then, we prove that this sequence is unique conditionally to the Poisson point measure determining the jumps in the population and to the stochastic flows corresponding to the dynamic of the trait. Finally, using Assumptions \ref{assu:debut1} and \ref{assu:debut2}, we prove that the number of individuals in the population does not blow up in finite time. This is detailed in \cite{marguet2016} (Theorem 2.3.).

\subsection{The trait of sampled individuals at a fixed time : Many-to-One formula\label{sampling}}

In order to characterize the trait of a uniformly sampled individual, the spinal approach (\cite{chauvin1988kpp},\cite{lyons1995conceptual}) consists in following a "typical" individual in the population whose behavior summarizes the behavior of the entire population. In this section, we give the dynamic of the trait of a typical individual: it is a Markov process called the auxiliary process.

From now on, we assume that for all $x\in\mathcal{X}$, $t\geq 0$ and $s\leq t$, $m(x,s,t)\neq 0$.  With a slight abuse of notation, we denote by $X_s^u$ the trait of the unique ancestor living at time $s$ of $u$.

Let $f$ be a non-negative measurable function and $x\in\mathcal{X}$. The idea is to notice that the following operator:
\begin{align}\label{eq:semigroup_auxi}
	P_{r,s}^{(t)}f(x):=\frac{\mathbb{E}\left[\sum_{u\in V_t} f\left(X_s^u\right)\big| Z_r=\delta_x\right]}{m(x,r,t)},
\end{align}
is a conservative (non-homogeneous) semi-group. Then, the auxiliary process is defined as the time-inhomogeneous Markov process with associated family of semi-groups $\left(P_{r,s}^{(t)},r\leq s\leq t\right)$. It describes the dynamic of the trait of a "typical" individual in the population in the sense that it satisfies a so-called "Many-to-One" formula. This formula is true under Assumptions \ref{assu:debut1}, \ref{assu:debut2} and under the following technical assumptions:
\begin{assumptionA}\label{assu:doeblin}
	There exists a function $C$ such that for all $j\leq k,\ j,k\in\mathbb{N}$ and $0\leq s\leq t$, we have:
	\[
	\sup_{x\in\mathcal{X}}\sup_{s\in[0,t]}\int_\mathcal{X}\frac{m(y,s,t)}{m(x,s,t)}P_j^{(k)}(x,dy)\leq C(t),\ \forall t\geq 0.
	\]
\end{assumptionA}
This assumption tells us that we control uniformly in $x$ the benefit or the penalty of a division.
\begin{assumptionA}\label{assu:differentiable}
	For all $t\geq 0$ and $x\in\mathcal{X}$ we have:
	\begin{itemize}
		\item[-]$s\mapsto m(x,s,t)$ is differentiable on $[0,t]$ and its derivative is continuous on $[0,t]$,
		\item[-] for all $f\in\mathcal{D}(\mathcal{A}):=\left\lbrace f\in\mathcal{D}(\mathcal{G})\text{ s.t. } m(\cdot,s,t)f\in\mathcal{D}(\mathcal{G})\ \forall t\geq 0,\ \forall s\leq t \right\rbrace$, $s\mapsto \mathcal{G}(m(\cdot,s,t)f)(x)$ is continuous,
	\end{itemize}
\end{assumptionA}
We can now state the Many-to-One formula.
\begin{theorem} Under Assumptions \ref{assu:debut1}, \ref{assu:debut2}, \ref{assu:doeblin} and \ref{assu:differentiable}, for all $t\geq 0$, for all $x_0\in\mathcal{X}$ and for all non-negative measurable functions $f:\mathcal{X}\rightarrow\mathbb{R}_+$, we have:
	\begin{equation}\label{mtomesurable}
	\E_{\delta_{x_0}}\left[\sum_{u\in V_{t}}f\left(X_{s}^{u}\right)\right]=m(x_0,0,t)\E_{x_0}\left[f\left(Y_{s}^{(t)}\right)\right],
	\end{equation}
	where $\left(Y_s^{(t)},s\leq t\right)$ is an inhomogeneous-Markov process with infinitesimal generator $\left(\mathcal{A}_s^{(t)},\mathcal{D}(\mathcal{A})\right)$ given for $f\in\mathcal{D}(\mathcal{A})$ and $x\in\mathcal{X}$ by:
	\begin{align*}
		\mathcal{A}_{s}^{(t)}f(x)= & \widehat{\mathcal{G}}_{s}^{(t)}f(x)
		+\widehat{B}_{s}^{(t)}(x)\int_{\mathcal{\mathcal{X}}}\left(f\left(y\right)-f\left(x\right)\right)\widehat{P}_{s}^{(t)}\left(x,dy\right),
	\end{align*}
	where:
	\begin{align*}
		\widehat{\mathcal{G}}_{s}^{(t)}f(x)=\frac{\mathcal{G}\left(m(\cdot,s,t)f\right)(x)-f\left(x\right)\mathcal{G}\left(m(\cdot,s,t)\right)(x)}{m(x,s,t)},
	\end{align*}
	\begin{align*}
		\widehat{B}_{s}^{(t)}(x)=B(x)\int_{\mathcal{X}}\frac{m(y,s,t)}{m(x,s,t)}m(x,dy),
	\end{align*}
	\begin{align*}
		\widehat{P}_{s}^{(t)}(x,dy)=\frac{m(y,s,t)}{m(x,s,t)}m(x,dy)\left(\int_{\mathcal{X}}\frac{m(y,s,t)}{m(x,s,t)}m(x,dy)\right)^{-1}.
	\end{align*}
\end{theorem}
\subsubsection*{Comments on the proof.}To prove the Many-to-One formula \eqref{mtomesurable}, we first show that the family of semi-groups $\left(P_{r,s}^{(t)},r\leq s\leq t\right)$ is uniquely defined as the unique solution of an integro-differential equation (see Lemma 3.2 in \cite{marguet2016}). In particular, we need Assumption \ref{assu:doeblin} to prove the uniqueness. Then, the infinitesimal generator $\left(\left(\mathcal{A}_{s}^{(t)}\right)_{s\leq t},\mathcal{D}\left(\mathcal{A}\right)\right)$ of the auxiliary process is obtained by differentiation of the semi-group $\left(P_{r,s}^{(t)},r\leq s\leq t\right)$ using Assumption \ref{assu:differentiable}.

\begin{remark}
	We can also prove a pathwise version of \eqref{mtomesurable} using a monotone class argument (Theorem 3.1. in \cite{marguet2016}).  
\end{remark}
Unlike previous works on this subject (\cite{georgii2003supercritical}, \cite{hardy2009spine}, \cite{cloez}), the existence of our auxiliary process does not rely on the existence of spectral elements for the mean operator of the branching process. In particular, we can apply this result to models with a varying environment. 
The dynamic of this auxiliary process heavily depends on the comparison between $m(x,s,t)$ and  $m(y,s,t)$, for $x,y\in\mathcal{X}$. It emphasizes several bias due to growth of the population. First, the auxiliary process jumps more than the original process, if jumping is beneficial in terms of number of descendants. This phenomenon of time-acceleration also appears for examples in \cite{chauvin1988kpp}, \cite{lyons1995conceptual} or \cite{hardy2009spine}. Moreover, the reproduction law favors the creation of a large number of descendant as in \cite{BDMV} and the non-neutrality favors individuals with an "efficient" trait at birth in terms of number of descendants. Finally, a new bias appears on the dynamic of the trait because of the combination of the random evolution of the trait and non-neutrality. Indeed, if the dynamic of the trait is deterministic, we have $\widehat{\mathcal{G}}_s^{(t)}f(x)=\mathcal{G}f(x)$.

\subsection{Ancestral lineage of a uniform sampling at a fixed time in a large population}\label{infiniteparticle}\label{sec:large}
The Many-to-One formula \eqref{mtomesurable} gives the law of the trait of a typical individual in the population. But so far, we did not prove that it corresponds to the trait of a uniformly sampled individual. In particular, we have to take into account that the number of individuals in the population is stochastic and depends on the dynamic of the trait. Using the law of large numbers, we can approximate the number of individuals in a population starting for $n$ individuals, divided by $n$, by the mean number of individual in the population. That is why we now look at the ancestral lineage of a uniform sampling in a large population.

It only makes sense to speak of a uniformly sampled individual at time $t$ if the population does not become extinct before time $t$. For all $t\geq 0$, let $\Omega_t=\left\lbrace N_t>0\right\rbrace$ denote the event of survival of the population. Let $\nu\in\mathcal{M}_P(\mathcal{X})$ be such that:
\begin{align*}
	\mathbb{P}_{\nu}(\Omega_t)>0.
\end{align*}
We set
\begin{align*}
	\nu_n:=\sum_{i=1}^n\delta_{X_i},
\end{align*}
where $X_i$ are i.i.d. random variables with distribution $\nu$.
For $t\geq 0$, we denote by $U(t)$ the random variable with uniform distribution on $V_t$ conditionally on $\Omega_t$ and by $\left(X^{U(t)}_s,\ s\leq t\right)$ the process describing the trait of a sampling along its ancestral lineage. 
If $X$ is a stochastic process, we denote by $X^{\nu}$ the process with initial distribution $\nu\in\mathcal{M}_P(\mathcal{X})$.
\begin{theorem}\label{th:grdpop} Under Assumptions \ref{assu:debut1},\ref{assu:debut2}, \ref{assu:doeblin} and \ref{assu:differentiable}, for any $t\geq 0$, the following convergence in law in $\mathbb{D}\left([0,t],\mathcal{X}\right)$ holds:
	\begin{align*}
		X_{[0,t]}^{U(t),\nu_n}\underset{n\rightarrow+\infty}{\longrightarrow}Y_{[0,t]}^{(t),\pi_t},\ \text{where }\pi_t(dx)=\frac{m(x,0,t)\nu(dx)}{\int_{\mathcal{X}}m(x,0,t)\nu(dx)}.
	\end{align*}
\end{theorem}

\begin{remark}
	If we start with $n$ individuals with the same trait $x$, we obtain:
	\begin{align*}
		\mathbb{E}\left[F\left(X_{[0,t]}^{U(t),\nu_n}\right)\right]\underset{n\rightarrow +\infty} {\longrightarrow}\mathbb{E}_x\left[F\left(Y_{[0,t]}^{(t)}\right)\right].
	\end{align*}
	Therefore, the auxiliary process describes exactly the dynamic of the trait of a uniformly sampled individual in the large population limit, if all the starting individuals have the same trait. 
	If the initial individuals have different traits at the beginning, the large population approximation of a uniformly sampled individual is a linear combination of the auxiliary process. 
\end{remark}

\section{Averaging for some integrate-and-fire models}

\subsection{The model}

We examine the following slow-fast hybrid version of integrate-and-fire models \cite{Quad}, used in mathematical neuroscience. In such a setting, $X$ represents the membrane potential of a neural cell which is increasing until it reaches some threshold $c$, corresponding to the time where a nerve impulse is triggered, and then the potential is reset to some lower value.  More precisely, the process $(X(t),t\in[0,T])$, with $T$ some time horizon, obeys the following dynamic:
\begin{enumerate}
	\item{\bf Initial state:} At time $T^\ast_0=0$, the process starts at $X(T^\ast_0)=\xi_0$, a random variable with support included in $(m,c)$ where $\{c\}$ is considered as a boundary and $m<c$ is some real.
	\item{\bf First jumping time:}  Let $Y$ be a continuous time Markov chain valued in a countable space ${\mathcal{Y}}$. This chain starts at $Y(0)=\zeta$, a $\mathcal{Y}$-valued random variable. The first hitting time of the boundary occurs at time $T^\ast_1$ defined as
	$$
	T^\ast_1=\inf\left\{t>0~;~\xi_0+\int_{T^\ast_0}^{t} \alpha(Y(s))F(X(s)) ds=c\right\},
	$$
	where $\alpha$ is a positive measurable function such that $\alpha(\mathcal Y)$ is bounded from above and $F$ is a positive continuous function.
	\item {\bf Piecewise deterministic motion:} For $t\in[T^\ast_0,T^\ast_1)$, we set
	$$
	X(t)=\xi_0+\int_{T^\ast_0}^{t} \alpha(Y(s))F(X(s)) ds.
	$$
	The dynamic of $X$ is thus continuous here, and given by the differential equation:
	$$
	\frac{dX}{dt}(t)=\alpha(Y(t))F(X(t)),\quad X(0)=\xi_0.
	$$
	\item {\bf Jumping measure:} Then, at time $T^{\ast -}_1$, the process is constrained to stay inside $(m,c)$ by jumping according to the $Y$-dependent measure $\mu_{Y_{T^{\ast -}_1}}$ whose support is included in $(m,c)$:
	$$
	\forall A\text{ Borel subset of }(m,c),\quad \mathbb{P}(X(T^\ast_1)\in A)=\mu_{Y_{T^{\ast -}_1}}(A).
	$$
	\item {\bf And so on:} Go back to step 1 in replacing $T^\ast_0$ by $T^\ast_1$ and $\xi_0$ by $\xi_1=X(T^\ast_1)$.
\end{enumerate}

\begin{figure}
	\begin{center}
		\begin{tikzpicture}
		\draw[thick,->](0,-1)--(0,5) node[left]{$X(t)$};
		\draw[very thick](-0.2,4.5) node[left]{$c$}--(7,4.5);
		\draw[thick] (0,2) node[left]{$\xi_0$}--(1,3)--(3,4)--(3.5,4.5);
		\draw[dashed,->] (3.5,4.5)--(3.5,3) node[left]{$\xi_1$};
		\draw[thick] (3.5,3)--(4.5,4)--(5.5,4.5);
		\draw[dashed,->] (5.5,4.5)--(5.5,1.5) node[left]{$\xi_2$};
		\draw[thick,->](0,-1)--(0,0.75) node[left]{$Y(t)$};
		\draw[thick,->](-0.2,-0.75)--(7,-0.75) node[below]{$t$}; 
		\draw[thick] (-0.1,0) --(1,0);
		\draw (-0.1,0) node[left]{$1$}-- (0,0);
		\draw[thick] (-0.1,-0.5) node[left]{$\frac12$}--(0,-0.5);
		\draw[thick] (1,-0.5)--(3,-0.5);
		\draw[thick] (3,0)--(4.5,0);
		\draw[thick] (4.5,-0.5)--(5.5,-0.5);
		\draw[gray,dashed] (1,-0.75)--(1,3);
		\draw[gray,dashed] (3,-0.75)--(3,4);
		\draw[gray,dashed] (4.5,-0.75)--(4.5,4);
		\draw (3.5,-0.7)--(3.5,-0.8) node[below]{$T^\ast_1$};
		\draw (5.5,-0.7)--(5.5,-0.8) node[below]{$T^\ast_2$};
		\end{tikzpicture}
		\caption{A trajectory of $X$, in a piecewise linear case, $dX(t)/dt=Y(t)$, with $Y$ switching between $1/2$ and $1$.}\label{fig:ex}
	\end{center}
\end{figure}
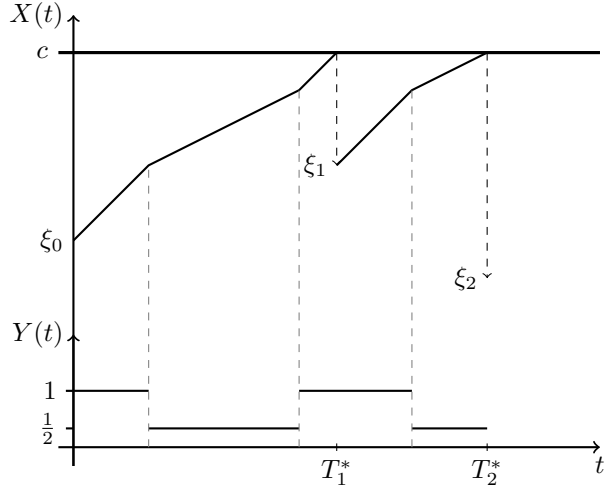

This is not hard to see that the process $(X,Y)$ is a piecewise deterministic Markov process. Moreover, due to the presence of a boundary $c$, this is also a constrained Markov process. Our aim is, in this quite simple situation, to describe the behavior of the process $X$ at the boundary $c$, when the dynamic of the underlying process $Y$ is infinitely accelerated.

\subsection{Acceleration and averaged process}

From now on, we assume that the underlying process of celerities, that is the continuous time Markov chain $Y$, has a fast dynamic, by introducing a (small) timescale parameter $\varepsilon $ such that
$$
\forall t\geq0,\quad Y_\varepsilon (t)=Y\left(t/\varepsilon \right).
$$
In the same time, to insure a limiting behavior, we assume that $Y$ is positive recurrent with intensity matrix $Q=(q_{zy})_{z,y\in\mathcal Y}$  and invariant probability measure $\pi$, considered as a row vector. For convenience, let us also define by $\alpha(V)$ the diagonal matrix such that $\text{diag}(\alpha(V))=\{\alpha(y)~;~y\in{\mathcal Y}\}$.\\
As $\varepsilon $ goes to zero, the process $Y_\varepsilon $ converges towards the stationary state associated to $Y$ in the sense that, by the ergodic theorem,
$$
\forall t\geq0,\forall y\in{\mathcal Y},\quad \lim_{\varepsilon \to0}\mathbb{P}(Y_\varepsilon (t)=y)=\pi(y).
$$
Therefore, as $\varepsilon $ goes to zero, the process $X_\varepsilon $, defined as $X$ by replacing $Y$ by $Y_\varepsilon $, should have its dynamic averaged with respect to the measure $\pi$. The behavior of the limiting process away from the boundary is indeed not hard to describe, see for example \cite{Kurtz92} and references therein. Here and later on, $\mathbb{D}\left(\left[0,\tau\right],\mathbb{R}\right)$ is the space of c\`adl\`ag functions from $[0,\tau]$ to $\mathbb{R}$, with $\tau$ a time horizon. For technical reasons, we assume that
\begin{itemize}
	\item $\frac1F$ is integrable over $(m,c)$,
	\item $\alpha({\mathcal Y})$ is bounded from above,
	\item $\mathbb{E}(\sup_{\varepsilon\in(0,1]}p^\ast_\varepsilon(T))$ is finite, where $p^\ast_\varepsilon(T)$ is the counting measure at the boundary for the process $X_\varepsilon$.
\end{itemize}
\begin{proposition}\label{prop;av:nb}
	Assume that $\xi_0$ is deterministic. Then, for any $\eta>0$, the process $X_\varepsilon $ converges in law towards a process $\bar X$ on $\mathbb{D}\left(\left[0,\frac{c-\xi_0}{\max\alpha(\mathcal Y)}-\eta\right],\mathbb{R}\right)$ defined as:
	\begin{align*}
		\bar X(t)&=\xi_0+\sum_{y\in\mathcal Y}\alpha(y)\pi(y)\int_0^tF(\bar{X}(s))d s.
	\end{align*}
\end{proposition}
Proposition \ref{prop;av:nb} gives the behavior of the limiting process away from the boundary: celerities are averaged against the measure $\pi$. But what happens at boundary? This is what we want to characterize. The following result holds.
\begin{theorem}\label{thm:av}
	The process $X_\varepsilon$ converges in law in $\mathbb{D}([0,T],\mathbb{R})$ towards a process $\bar X$ such that for any sufficiently regular function $f:(m,c)\to \mathbb{R}$ the process
	\begin{align}\label{eq:pb:mart:av}
		f(\bar X(t))-f(\xi_0)-\sum_{y\in\mathcal Y}\alpha(y)\pi(y)\int_0^t f'(\bar X(s))F(\bar X(s)) d s-\int_0^t \int_{m}^c [f(u)-f(\bar X(s^-))]\bar\mu(d u) \bar p^\ast(d s)
	\end{align}
	for $t\in[0,T]$, defined a martingale, with $\bar p^\ast$ the counting measure at the boundary for $\bar X$. The averaging measure at the boundary $\bar\mu$ is defined by
	$$
	\bar\mu(d u)=\sum_{y\in{\mathcal Y}}\mu_{y}(d u)\pi^\ast(y)
	$$
	where, for $y\in\mathcal{Y}$, $\pi^*$ is given by 
		$$
		\pi^*(y)=\frac{\pi(y)\alpha(y)}{\sum_{y'\in\mathcal Y}\pi(y')\alpha(y')}.
		$$
\end{theorem}
We can read in (\ref{eq:pb:mart:av}) the behavior of the averaged process $\bar X$. Away from the boundary, the dynamic of $\bar X$ is the one of the averaged flow, as indeed described in Proposition \ref{prop;av:nb}. At the boundary, the averaged jump measure is not given by the averaged version of the jumping measure, that is
$$
\bar\mu(d u)\neq \sum_{y\in{\mathcal Y}}\mu_{y}(d u)\pi(y),
$$
but the values of the celerity at the boundary are taken into account through the presence of $\alpha$. There is a balance between the celerity intensities and the invariant probabilities. For example, consider a piecewise linear motion whose speed switches at rate $1/\varepsilon$ between $1$ and $2$. When, $\varepsilon$ goes to zero, the process will spend as much time increasing with celerities $1$ and $2$, but when increasing with celerity $2$, the process obviously increases twice more than with celerity one in a same time window, and then approaches the boundary twice more rapidly, and therefore will have in fact twice more chance to hit the boundary with this celerity.\\
Note that because of the separation of variables in the form of the flow, the value of the process $\bar X$ at the boundary does not appear in the expression of $\pi^\ast$, as it could be the case in more general situations. Theorem such as \ref{thm:av} can be derived using the Prohorov program : prove the tightness of the family $\{X_\varepsilon,\varepsilon\in (0,1)\}$ and then identify the limit. Tightness for constrained Markov processes, as the process described in the present section, has an interesting and rich literature, see \cite{Kurtz90} and references therein. We refer to \cite{AG} for more details about the proof of Theorem \ref{thm:av}.

\section{Stochastic models of protein production with cell division and gene replication}


The production of proteins is the main mechanism of the prokaryotic cells (like
bacteria). These functional molecules represent about $3.6\times10^{6}$
molecules of $2000$ different types and compose about half of the dry mass
of the cell \cite{bremer_modulation_1996} and this pool needs to
be duplicated in the cell cycle time. In total, it is estimated that
$67\%$ of the energy of the cell is dedicated to the protein production
\cite{russell_energetics_1995}. The protein production is a two steps
process by which the genetic information is first transcribed into an
intermediate molecule, the messenger-RNA (mRNA), and then translated
into a protein. Experimental studies have shown
that this production is subject to a large variability \cite{elowitz_stochastic_2002,ozbudak_regulation_2002,taniguchi_quantifying_2010}.
This heterogeneity can either be explained by the stochastic nature
of the production mechanism (both transcription and translation are due to
random encounter between molecules) or the effect of cell scale phenomenon
like the DNA replication, cell division, or fluctuation in the resources.

In particular, the extensive experimental study of Taniguchi~et~al.~(2010)
\cite{taniguchi_quantifying_2010} measured the variability of protein
concentration for more than $1000$ different proteins in \emph{E.
	coli}. For each of them, the authors have measured the average protein
concentration, the average mRNA concentration and their lifetime.
They interpret different sources of proteins, either from the protein
production mechanism itself (i.e. the transcription of mRNAs and the
translation of proteins) or from external factors, like the gene replication,
division or the sharing of common resources in the production.

One can interpret the experimental results in the light of classical
stochastic models of protein production \cite{berg_model_1978,rigney_stochastic_1977,paulsson_models_2005}
that predict the variability of protein production. But these classical
models do not consider important phenomenon that may have an impact
on the production: they usually do not consider the replication of
the gene at some point in the cell cycle and also do not take into
account the random assignment of each protein in either of the two
daughter cells at division. Moreover, they only consider the number
of mRNAs and proteins without any explicit notion of volume, and therefore
the notion of concentration is unclear in this case. Since the experimental
study only considers concentrations, it would be difficult to have
quantitative comparisons between the variances predicted by these
models and the ones experimentally measured.

We therefore propose here a more realistic model of protein production
that takes into account all the basic features that can be expected
for the production of a type of protein inside a cell cycle: the transcription,
the translation, the gene replication and the division. We also explicitly
represent the number of mRNAs and proteins in the growing volume of
the bacteria, so that we can consider the concentration of each of
these entities in the cell. We will then estimate the theoretical contribution
of these features to the protein variance through mathematical analysis. We
will also make a biological interpretation of these results with respect to
experimental measures.

\subsection{Model and Results}

Our model presented here is adapted from a simple gene constitutive model that
does not consider any gene regulation (it is a particular case of the model
presented in \cite{paulsson_models_2005}). Like this classical model, our model
is gene centered (it represents the production of only one protein), and has two
stages (transcription and translation) where both the number of mRNAs, $M$, and
proteins, $P$, are explicitly represented. Each event (creation or degradation
of an mRNA, creation of a protein, etc.) is supposed to occur at random times.
The time intervals are considered as exponentially distributed. In addition to
this classical framework, we introduce a notion of cell cycle represented by
cell growth, division and gene replication. Periodically two events occurs: the
replication at a deterministic time $\tau_{R}$ in the cell cycle that doubles the
mRNA production rate and the division after a time $\tau_{D}$ starting a new
cell cycle. We can then consider the evolution
of one type of mRNA and protein through a cell lineage (see Figure~\ref{fig:Cell-cycle}).

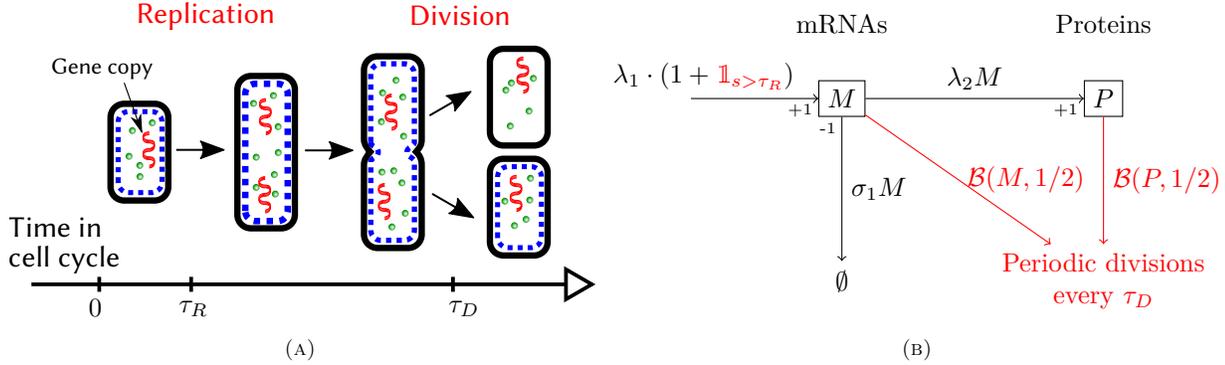
\begin{figure}
	\begin{centering}
		\subfloat[\label{fig:Cell-cycle}]{\begin{tikzpicture}
			\node[anchor=south west,inner sep=0] (im) at (0,0) {\includegraphics[height=4cm]{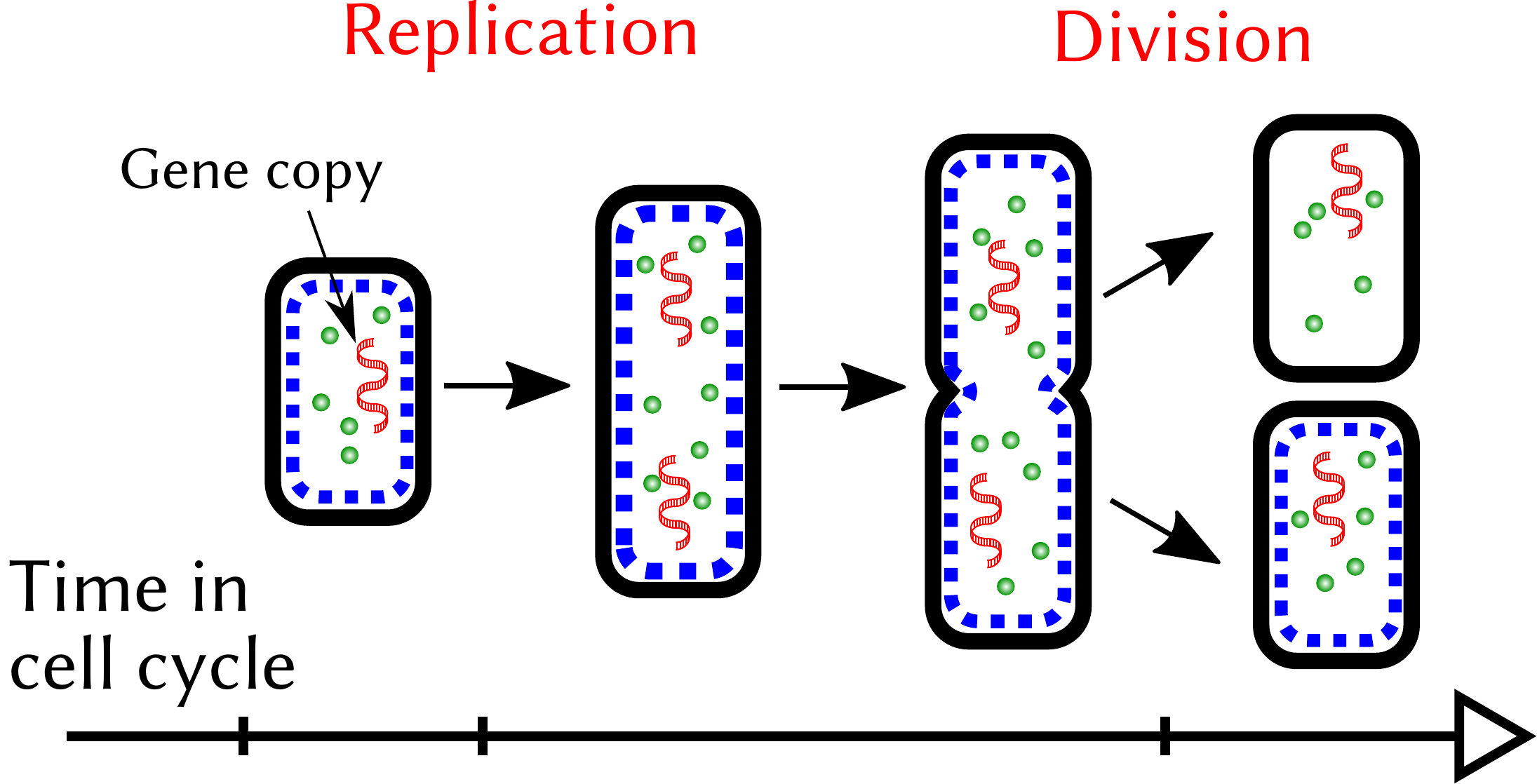}};
			\draw (im.south west)+(1.2,-0.3) node[above]{$0$};
			\draw (im.south west)+(2.5,-0.3) node[above]{$\tau_R$};
			\draw (im.south west)+(6.1,-0.3) node[above]{$\tau_D$};
			\end{tikzpicture}}~~\subfloat[\label{fig:model-with-replication}]{\centering{}\begin{tikzpicture}
			\matrix [column sep=17mm, row sep=17mm,ampersand replacement=\&] {
				\node (g) [draw=none,fill=none] {}; \&
				\node (m) [draw, shape=rectangle] {$M$}; \&
				\node (p) [draw, shape=rectangle] {$P$}; \\
				\&
				\node (em) [draw=none,fill=none] {$\emptyset$}; \&
				\node (ep) [draw=none,fill=none,color=red,align=center] {
					Periodic divisions \\  every $\tau_{D}$}; \\
			};
			\draw[-> ] (g) -- (m) node[pos=0.1,above] {$\lambda_1\cdot (1+{\color{red} \ind_{s>\tau_R}})$};
			\draw[-> ] (m) -- (em) node[right,midway] {$\sigma_1M$};
			\draw[-> ] (m) -- (p) node[midway,above] {$\lambda_2M$};
			\draw[-> ,red] (m) -- (ep) node[right,midway] {${\mathcal B}(M,1/2)$};
			\draw[-> ,red] (p) -- (ep) node[right,midway] {${\mathcal B}(P,1/2)$};
			
			\draw (m.north)+(0,0.5) node[above]{mRNAs};
			\draw (p.north)+(0,0.5) node[above]{Proteins};
			
			\draw (m.west) node[below left,inner sep=2] {\tiny +1};
			\draw (m.south) node[below left,inner sep=2] {\tiny -1};
			\draw (p.west) node[below left,inner sep=2] {\tiny +1};
			
			\end{tikzpicture}}
		\par\end{centering}
	\caption{Model with gene replication and division. \textbf{(A)} The concept
		of cell cycle in the model. \textbf{(B)} Model of production for one
		protein.}
\end{figure}

Overall, the model represents five aspects that intervene in protein
production (see Figure~\ref{fig:model-with-replication}):
\begin{description}
	\item [{mRNAs}] Messenger-RNAs are transcribed at rate $\lambda_{1}$ before
	the replication and at rate $2\lambda_{1}$ after gene replication.
	When transcription happens, the number $M$ of mRNAs  increases by
	$1$. As in the previous model, each mRNA has a lifetime of rate $\sigma_{1}$
	(so the global mRNA degradation rate is $\sigma_{1}M$).
	\item [{Proteins}] Each mRNA can be translated into a protein at rate $\lambda_{2}$
	(so the global protein production rate is $\lambda_{2}M$). The number
	of proteins $P$ is increased by $1$. As the protein lifetime is
	usually of the order of several cell cycles, we consider that the
	proteins do not degrade.
	\item [{Gene~replication}] After a deterministic time $\tau_{R}$ after
	each division (with $\tau_{R}<\tau_{D}$), the gene replication occurs.
	The gene responsible for the mRNA transcription is replicated, hence
	doubling the mRNA transcription rate until next division.
	\item [{Division}] Divisions occur periodically after deterministic time
	interval $\tau_{D}$. After each division, we only follow one of the two
	daughter cells (see Figure~\ref{fig:Cell-cycle}) so that each mRNA and each
	protein have a probability $1/2$ to be in this next considered cell. The
	result is a \emph{binomial sampling}: for instance, with $M$ mRNAs just
	before division, the number of mRNAs after division follows a binomial
	distribution ${\mathcal B}\left(n,p\right)$ with parameters $n=M$ and
	$p=1/2$. Moreover, as there is only one copy of the gene in the newborn
	cell, the mRNA transcription rate is anew set to $\lambda_{1}$ until
	the next gene replication.
	\item [{Volume~growth}] The volume $V(s)$ considers the entire volume
	of the cell at any moment $s$ of the cell cycle and it increases
	as the cell grows. One considers the volume growth of the cell as
	deterministic. In real life experiments, a bacteria volume globally
	grows exponentially (see \cite{wang_robust_2010}) and approximately
	doubles its volume at the time of division $\tau_{D}$. As a consequence,
	in our model, for a time $s$ in the cell cycle, then the volume grows
	as
	\[
	V(s)=V_{0}2^{s/\tau_{D}}
	\]
	with $V_{0}$ the typical size of a cell at birth.
\end{description}
From this, if $M_{s}$ and $P_{s}$ denote respectively the number
of mRNAs and proteins at time $s$, the concentrations are
\[
\frac{M_{s}}{V(s)}\mbox{ and }\frac{P_{s}}{V(s)}\mbox{.}
\]

Our main results are the analytical expressions for the distribution of the number of mRNAs and the first moments of the number of proteins
at any time $s$ of the cell cycle, when the system is at equilibrium
(ie after many cell divisions).
\begin{proposition}
	\label{thm:xs_rep}At equilibrium, the mRNA number $M_{s}$ at time
	$s$ in the cell cycle follows a Poisson distribution of parameter
	\[
	x_{s}=\frac{\lambda_{1}}{\sigma_{1}}\left[1-\frac{e^{-(s+\tau_{D}-\tau_{R})\sigma_{1}}}{2-e^{-\tau_{D}\sigma_{1}}}+\ind_{s\geq\tau_{R}}\left(1-e^{-(s-\tau_{R})\sigma_{1}}\right)\right]\mbox{.}
	\]
\end{proposition}
The proof of this proposition uses the framework of Marked Point Poisson
Processes to determine the mRNA distribution at any time of the cell
cycle knowing the random variable $M_{0}$, the number of mRNA at
birth. Then, the distribution of $M_{0}$ can be determined as the
equilibrium stipulates that the distribution of $M$ is the same at
time $0$ and just after the next division. The proof will be given
in \cite{dessalles_stochastic_2017} (in preparation).

\begin{proposition}
	\label{prop:EP_VarP}At any time $s\in[0,\tau_{R}[$ before replication,
	the mean and the variance of $P_{s}$ are given as a function of the
	moments of $\left(P_{0},M_{0}\right)$ by:
	\begin{eqnarray*}
		\Esp{P_{s}} & = & \Esp{P_{0}}+\lambda_{2}\frac{\lambda_{1}}{\sigma_{1}}s+\left(x_{0}-\frac{\lambda_{1}}{\sigma_{1}}\right)\frac{1-e^{-\sigma_{1}s}}{\sigma_{1}},\\
		\V{P_{s}} & = & \V{P_{0}}+2\lambda_{2}\frac{1-e^{-\sigma_{1}s}}{\sigma_{1}}\Cov{P_{0},M_{0}}\\
		&  & \left(\lambda_{2}\frac{1-e^{-\sigma_{1}s}}{\sigma_{1}}\right)^{2}x_{0}+x_{0}\frac{\lambda_{2}}{\sigma_{1}}\left(1-e^{-\sigma_{1}s}+\frac{\lambda_{2}}{\sigma_{1}}\left[1-e^{-\sigma_{1}s}\left(e^{-\sigma_{1}s}+2s\sigma_{1}\right)\right]\right)+\\
		&  & +\frac{\lambda_{1}\lambda_{2}}{\sigma_{1}^{2}}\left[s\sigma_{1}-1+e^{-\sigma_{1}s}+2\frac{\lambda_{2}}{\sigma_{1}}\left(\sigma_{1}s\left(1+e^{-\sigma_{1}s}\right)-2\left(1-e^{-\sigma_{1}s}\right)\right)\right].
	\end{eqnarray*}
	with $x_{0}$ defined in Proposition~\ref{thm:xs_rep}.
\end{proposition}
Moreover, we can show a similar result for $s\in[\tau_{R},\tau_{D}[$,
a time after gene replication. Then, at equilibrium, we have that the
distribution of the couple $\left(M,P\right)$ after division is the
same as at the beginning of the cell cycle. This allows us to determine
the explicit values for $\Esp{P_{0}}$, $\V{P_{0}}$ and $\Cov{P_{0},M_{0}}$.
These complementary results and the proof of Proposition~\ref{prop:EP_VarP}
will be available in \cite{dessalles_stochastic_2017} (in preparation).

\subsection{Discussion}

In order to have realistic values for the parameters $\lambda_{1}$,
$\sigma_{1}$, $\lambda_{2}$ and $\tau_{R}$, we fix them to correspond to the average production of each protein measured
in Taniguchi~et~al.~(2010) \cite{taniguchi_quantifying_2010}.
It gives groups of parameters that represent a large spectrum of different
genes. Then, using Propositions \ref{thm:xs_rep} and \ref{prop:EP_VarP},
as well as the additional results, we can compute the concentration for
every gene at any time of the cell cycle.

\begin{figure}[h]
	
	\begin{centering}
		\includegraphics[width=6cm]{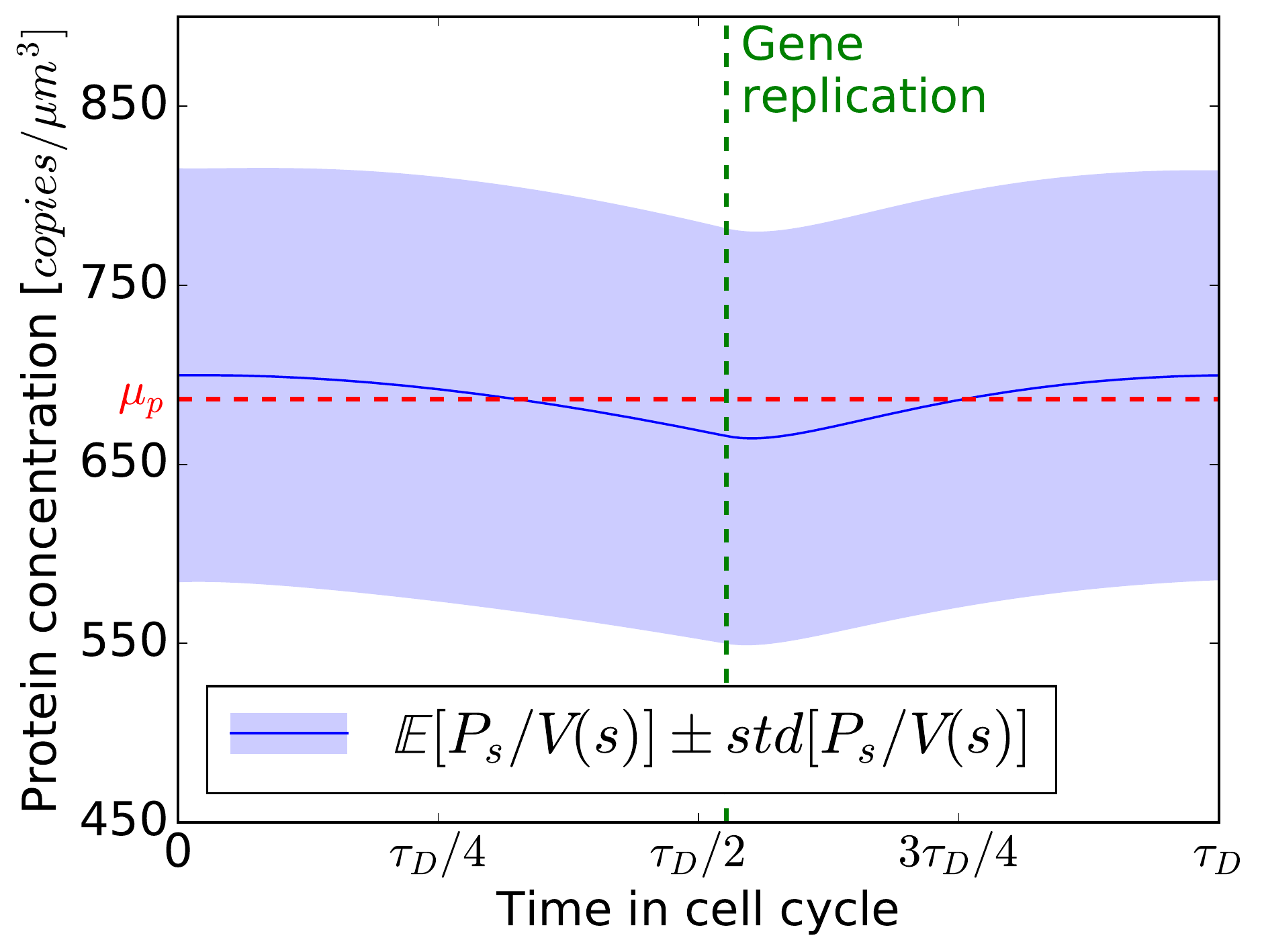}\caption{\label{fig:Profile-example}Profile example of the gene Adk, the central
			line represents the mean protein concentration over the cell cycle,
			and the coloured areas represent the standard-deviation in the models.}
		\par\end{centering}
\end{figure}

An example of the evolution of one protein concentration is shown
in Figure~\ref{fig:Profile-example}. It appears that the mean concentration
at a given time $s$ of the cell cycle $\Esp{P_{s}/V(s)}$ is not constant
during the cell cycle. The curve of $\Esp{P_{s}/V(s)}$ fluctuates around
$2\%$ of the global average protein production (denoted by $\mu_{p}$
in the graph). Experimental measures of average protein expression
during the cell cycle show similar results: for instance, the expression
of genes at different positions on the chromosome have been measured
in \cite{walker_generation_2016}; it shows a similar profile during
the cell cycle and depicts a fluctuation also around $2\%$ of the
global average (see Figure 1.d and Figure S6.b of the article).

Overall, for all the genes considered, the effect of the cell cycle
(the fluctuation of $\Esp{P_{s}/V(s)}$ around $\mu_{p}$) is small
compared to the standard deviation at each moment of the cell cycle
$\sqrt{\V{P_{s}/V(s)}}$ (the coloured areas in Figure~\ref{fig:Profile-example}).
We have determined the range of parameters where, on the contrary,
the fluctuations due to the cell cycle would be higher than the inherent
variability of the system. It appears that such regime is possible
only for highly active gene (high $\lambda_{1}$) and for mRNAs not
very active (low $\lambda_{2}$) and which exist for short periods
of time (high $\sigma_{1}$). Biologically, such regime seems unrealistic
because the cost of mRNA production would be very high.

\begin{figure}
	\centering{}\subfloat[\label{fig:tani}Coefficient of variation in Taniguchi et al. (2010)]{\includegraphics[width=6cm]{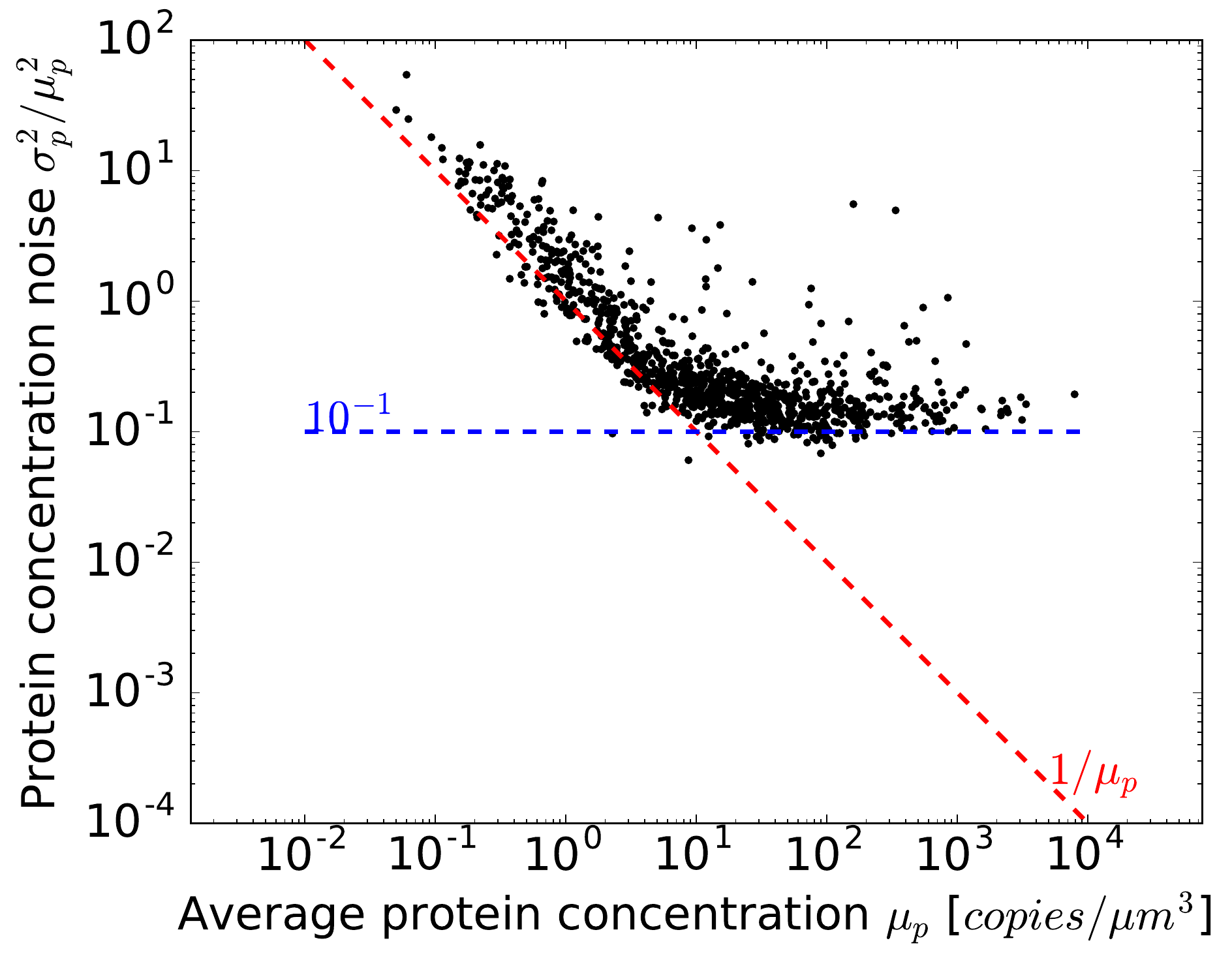}}~~\subfloat[Coefficient of variation in our model]{
		
		\includegraphics[width=6cm]{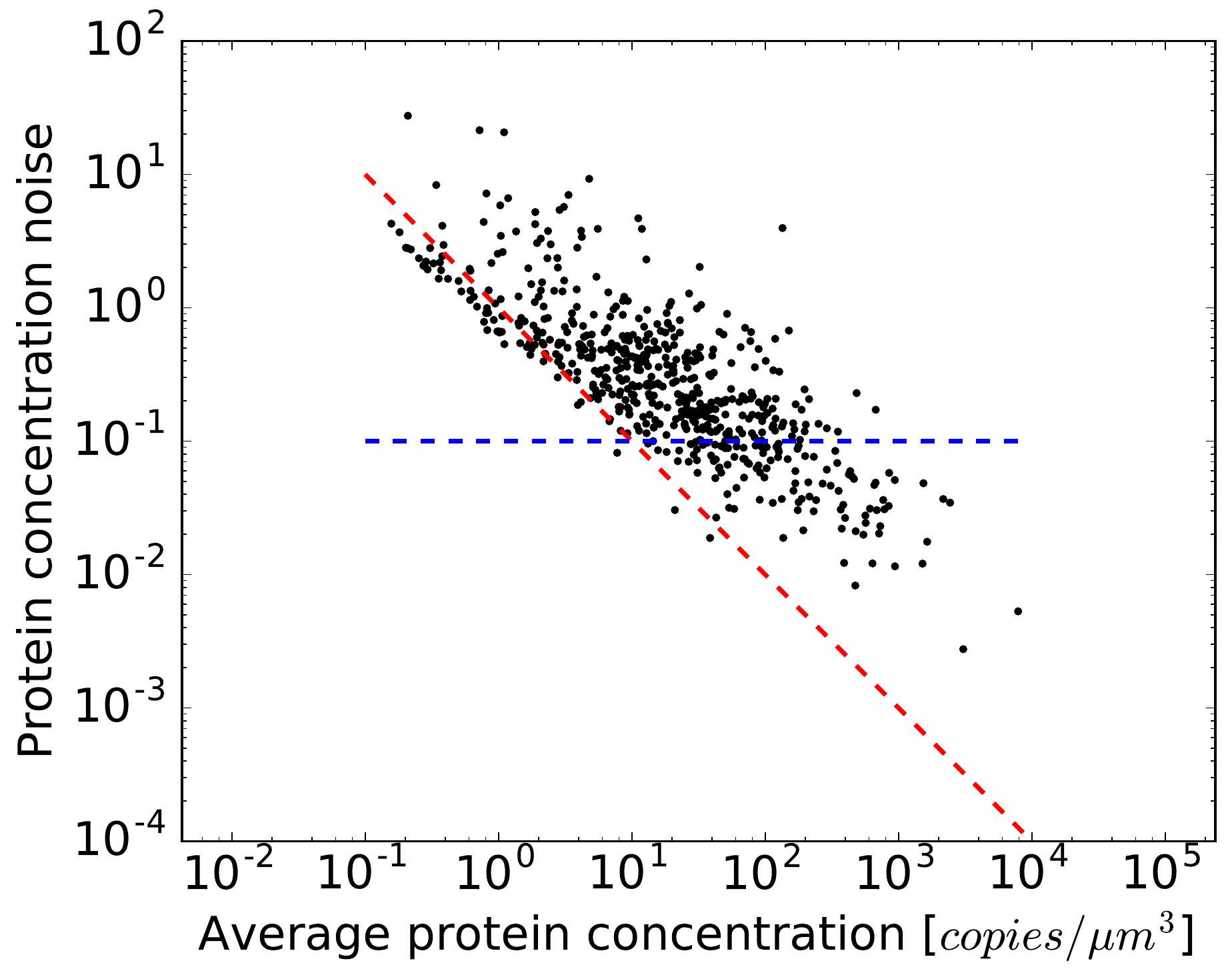}
		
	}\caption{\label{fig:Result_model3_Simulated-noise}Coefficient of variation
	of each proteins. \textbf{(A)} Each dot represents the coefficient
	of variation (CV) of the concentration for a particular protein (the
	CV is defined as $\sigma_{p}^{2}/\mu_{p}^{2}$ with $\sigma_{p}^{2}$
	the variance and $\mu_{p}$ the mean) experimentally measured in \cite{taniguchi_quantifying_2010}.
	It clearly appears two regimes: the first one (for $\mu_{p}<10$)
	shows a CV inversely proportional to the mean; the second regime (for
	$\mu_{p}>10$) shows a CV independent of the average production. \textbf{(B)}
	Same graph obtained with the proteins of our model. Even if the CV
	is roughly inversely proportional to the average production, there
	is no lower bound limit for highly expressed proteins.}
\end{figure}

In Figure~\ref{fig:Result_model3_Simulated-noise}, we compare the experimental
results obtained in the article of Taniguchi~et~al.~(2010) \cite{taniguchi_quantifying_2010}
and our analytical expressions obtained in our model, for the global coefficient of variation (CV) of protein concentration
as a function of the average concentration. In the experimental
measures, it clearly appears two regimes for the CV of protein concentration:
for low expressed proteins (mean protein concentration $<10$), the
CV roughly scales inversely with the average concentration; for genes
with higher protein production (mean protein concentration $>10$),
the CV becomes independent of the average protein production level,
the plateau is around $10^{-1}$. In our model, the global tendency
of the CV approximately inversely scales the average protein concentration,
and there is no lower bound limit, as in the experiments.

We have therefore shown through this model that both DNA replication
and division can not explain the second regime observed in Figure~\ref{fig:tani}.
Yet the authors of the experimental study propose another possible
origin for the noise observed in highly produced proteins; they propose
that it is due to fluctuations in the availability of RNA-polymerases
and ribosomes. These two macro-molecules are shared among the productions
of all the different proteins because they respectively participate
to the transcription of mRNAs and the translation of proteins. Since
they are present in limited quantities, fluctuations in their availability
could indeed have repercussions on the protein noise. We will present
in \cite{dessalles_stochastic_2017} an extension of the model presented
here that takes into account this sharing of ribosomes and RNA-polymerases
in the production of all proteins of the bacteria, and the impact
it has on the protein noise.


\bigskip

\textbf{Acknowledgements}
Bertrand Cloez would like to thank Romain Yvinec for his enthusiastic and stimulating organization of the "Journ\'ees MAS 2016" and Florent Malrieu for many interesting discussions on the subject of his section. The research of Bertrand Cloez was partially supported by the Agence Nationale de la Recherche PIECE 12-JS01-0006-01 and by the Chaire "Mod\'elisation Math\'ematique et Biodiversit\'e " .  The research of Aline Marguet was partially supported by the Chaire "Mod\'elisation
Math\'ematique et Biodiversit\'e".
Renaud Dessalles would like to thank Philippe Robert and Vincent Fromion for their support.

\bibliographystyle{plain}
\bibliography{biblio_pdmp_mas_2016}
\end{document}